\documentclass [12pt, leqno]{article}
\usepackage{amsthm}
\usepackage{extpfeil}

\newlength{\abstractwidth}
\renewcommand{\thanks}[1]{\footnote{#1}} 

\newcommand{\be}{\begin{equation}}
\newcommand{\bea}{\begin{eqnarray}}
\newcommand{\eea}{\end{eqnarray}} \newcommand{\ee}{\end{equation}}
 
 \def\ba{\begin{eqnarray}}
\def\ea{\end{eqnarray}}

\def\C{{\bf C}}

\def\o{\omega}

\def\o{\omega}

\def\al{\alpha}
\def\b{\beta}

\def\o{\omega}

\def\ti{\tilde}

\def\C{{\bf C}}

\def\cO{{\cal O}}

\def\cO{{\cal O}}

\def\[{{\bf [}}
\def\]{{\bf ]}}

\def\pl{\partial}

\begin{document}
\newtheorem{theorem}{Theorem}
\newtheorem{proposition}{Proposition}
\newtheorem{lemma}{Lemma}
\newtheorem{corollary}{Corollary}
\newtheorem{definition}{Definition}
\newtheorem{conjecture}{Conjecture}
\newtheorem{example}{Example}
\newtheorem{claim}{Claim}

\begin{centering}
 
\textup{\LARGE\bf The limit of the Yang-Mills flow on semi-stable bundles}

\vspace{5 mm}

\textnormal{\large Adam Jacob}

\begin{abstract}
{\small By the work of Hong and Tian it is known that given a holomorphic vector bundle $E$ over a compact K\"ahler manifold $X$, the Yang-Mills flow converges away from an analytic singular set. If $E$ is semi-stable, then the limiting metric is Hermitian-Einstein and will decompose the limiting bundle $E_\infty$ into a direct sum of stable bundles. Bando and Siu prove $E_\infty$ can be extended to a reflexive sheaf $\hat E_\infty$ on all of $X$. In this paper, we construct an isomorphism between $\hat E_\infty$ and the double dual of the stable quotients of the graded Seshadri filtration $Gr^s(E)^{**}$ of  $E$.  }

\end{abstract}

\end{centering}

\begin{normalsize}

\section{Introduction}
The purpose of this paper is to investigate the limiting properties of the Yang-Mills flow in the case where the flow does not converge. Given a holomorphic vector bundle $E$ over a compact K\"ahler manifold $X$, the Yang-Mills flow provides a heat flow approach to the Hermitian-Einstein problem, which was first solved in the case of curves by Narasimhan and Seshadri $\cite{NS}$, by Donaldson for algebraic surfaces $\cite{Don1}$, and then by Uhlenbeck and Yau for arbitrary dimension $\cite{UY}$. While Uhlenbeck and Yau used the elliptic approach of the method of continuity, Donaldson introduced a heat flow related to the Yang-Mills flow, and showed convergence of the flow is equivalent to an algebraic stability condition. This heat flow approach has been extended to various other settings, most notably to higher dimensional algebraic manifolds by Donaldson $\cite{Don2}$, to Higgs bundles by Simpson $\cite{Simp}$, and to reflexive sheaves by Bando-Siu $\cite{BaS}$. The latter two proofs require the $C^0$ estimate proven in $\cite{UY}$. There is also an exposition of the heat flow proof by Siu for all dimensions in $\cite{Siu}$. We say $E$ is stable (in the sense of Mumford-Takemoto) if for every proper coherent torsion-free subsheaf ${\cal F}\subset E$,
\be
\frac{deg({\cal F})}{rk({\cal F})}<\frac{deg(E)}{rk(E)}.\nonumber
\ee
In the stable case, the flow converges to a Hermitian-Einsten metric on $E$. While this theorem is extremely powerful and useful, it leaves open many interesting questions about what happens when the flow does not converge smoothly on all of $X$.

In the case where $E$ is not stable and $X$ is a K\"ahler surface, the convergence properties of the flow have been addressed by Daskalopoulos and Wentworth in $\cite{DW}$ and $\cite{DW2}$. They prove that away from from an analytic bubbling set and along a subsequence, the Yang-Mills flow converges to a limiting connection on the stable quotients of the graded Harder-Narasimhan-Seshadri filtration of $E$. We denote this direct sum of quotients as $Gr^{hns}(E)$, and the induced limiting connection on each stable quotient is smooth and Hermitian-Einstein. Furthermore, Daskalopoulos and Wentworth identified the bubbling set as being equal to the singular set of the torsion-free sheaf $Gr^{hns}(E)$, and were able to prove that the limiting bundle on which the flow converges is isomorphic to $Gr^{hns}(E)^{**}$, verifying a conjecture of Bando and Siu in this setting. 

This paper gives some partial results towards extending the work of Daskalopoulos and Wentworth to higher dimensions. In particular we restrict to the case where the bundle $E$ is semi-stable, which is defined by the condition that for every proper coherent torsion-free subsheaf ${\cal F}\subset E$,
\be
\frac{deg({\cal F})}{rk({\cal F})}\leq\frac{deg(E)}{rk(E)}.\nonumber
\ee
Given a subsequence of connections $A_j$ along the Yang-Mills flow, we define the analytic bubbling set by:
\be
Z_{an} =\bigcap_{r>0}\{x\in X\,|\liminf_{j\rightarrow\infty}\,r^{4-2n}\int_{B_r(x)}|F_{A_j}|^2\o^n\geq \epsilon\}.\nonumber
\ee
This set is the same singular set used by Hong and Tian in $\cite{HT}$. Our complete result is as follows:
\begin{theorem}
\label{main theorem}
Let $E$ be a semi-stable holomorphic vector bundle over a compact K\"ahler manifold $X$. Let $A(t)$ be a connection on $E$ evolving along the Yang-Mills flow. Then there exists a subsequence of times $t_j$ such that on $X\backslash Z_{an}$, the sequence $A(t_j)$ converges (modulo gauge transformations) in $C^\infty$  to a limiting connection $A_\infty$ on a limiting bundle $E_\infty$. $E_\infty$ extends to all of $X$ as a reflexive sheaf $\hat E_\infty$ which is isomorphic to the double dual of the stable quotients of the graded Seshadri filtration $Gr^s(E)^{**}$ of $E$.
\end{theorem}

Hong and Tian prove in  $\cite{HT}$ that a subsequence along the Yang-Mills flow $A_j$ converges smoothly to a limiting Yang-Mills connection on a limiting bundle $E_\infty$, away from the analytic bubbling set. They also prove that $Z_{an}$ is a holomorphic subvariety of $X$, although we do not utilize this result. By the work of Bando and Siu $\cite{BaS}$, we know $E_\infty$ extends to all of $X$ as a reflexive sheaf $\hat E_\infty$. The contribution of this paper is to construct an isomorphism between $\hat E_\infty$ and $Gr^s(E)^{**}$, verifying a conjecture of Bando and Siu from $\cite{BaS}$ in the case where $E$ is semi-stable. This result is attractive not only in that it gives a better understanding of the limiting properties of the flow, but also because $Gr^s(E)$ is canonical. Thus in the semi-stable case, the Yang-Mills flow does indeed ``break up" the bundle in a natural way that only depends on $E$.

Here we briefly describe the proof of Theorem $\ref{main theorem}$. Many of the main ideas can be found in $\cite{Buch}$,$\cite{DW}$, and $\cite{Don1}$. For us, the key tool is the existence of an approximate Hermitian-Einstein structure on $E$ from $\cite {J}$. This means that for all $\epsilon>0$, there exists a metric $H$ on $E$ with curvature $F$ such that
\be
\sup_X|\Lambda F-\mu I|<\epsilon,\nonumber
\ee
where $\mu$ is a fixed constant. Now, with a modification of the Chern-Weil formula we can use this approximate Hermitian-Einstein structure to show that along the Yang-Mills flow the second fundamental form of any destabilizing subsheaf $S$ of $E$ must go to zero, creating a holomorphic splititng of the limiting bundle $E_\infty=S_\infty\oplus Q_\infty$.  Assume $S$ has the lowest rank among all destabilizing subsheaves of $E$. Then $S$ is stable. We show that after normalization, the sequence of holomorphic inclusion maps $f_j:S\longrightarrow E$ (generated by the flow) converge on compact subsets and along a subsequence to a non-trivial holomorphic map $f_\infty$. The key difficulty is we loose control of the connections $A_j$ as we approach $Z_{an}$, forcing us to work delicately on compact subsets away from this singular set in order to get the proof to work.
Next, because $S$ is stable and $S_\infty$ semi-stable, it follows that the holomorphic map we constructed must in fact be an isomorphism. Applying this argument inductively to the quotients of the Seshadri filtration allows us to construct an isomorphism between $Gr^s(E)$ and $E_\infty$ on $X\backslash Z_{an}$. Theorem $\ref{main theorem}$ follows from the uniqueness of the reflexive extention $\hat E_\infty$. 

We remark that in order to give a complete generalization of the work of Daskalopoulos-Wentworth in the semi-stable case, one would need to show that in fact $Z_{an}$ is equal to the set where $Gr^s(E)$ fails to be locally free. To do so, a more detailed analysis of the singular set is needed.

The paper is organized as follows. In Section 2 we describe our setup and outline the relationship between the Yang-Mills flow and another geometric flow, which we call the Donaldson heat flow. In Section 3 we prove two convergence results along the flow, first the convergence of the Seshadri filtration and then the convergence of the holomorphic inclusion maps. Finally, in Section 4 we prove Theorem $\ref{main theorem}$. As a corollary we prove that the algebraic singular set of $Gr^s(E)$ is contained in the analytic bubbling set.

\medskip
\begin{centering}
{\bf Acknowledgements}
\end{centering}
\medskip

First and foremost, the author would like to thank his thesis advisor, D.H. Phong, for all his guidance and support during the process of writing this paper. The author would also like to thank Thomas Nyberg, Tristan Collins, and Valentino Tosatti for many enlightening discussions and encouragement. Finally, the author would like to express his upmost gratitude to Richard Wentworth for encouraging him to work on this problem. Professor Wentworth provided valuable insight into the structure of the proof, and for this the author is most grateful. This research was funded in part by the National Science Foundation, Grant No. DMS-07-57372, as well as Grant No. DMS-1204155. The results of this paper are part of the author's Ph.D. thesis at Columbia University.

\section{Background}

\label{bkg}

We begin with some basic facts about holomorphic vector bundles, and then introduce the Donaldson heat flow and the Yang-Mills flow.

 Let $E$ be a holomorphic vector bundle over a compact K\"ahler manifold $X$ of complex dimension $n$. Locally the K\"ahler form is given by:  
\be
 \o=\frac{i}{2}\,g_{\bar kj}\,dz^j\wedge d\bar z^k,\nonumber
 \ee
where $g_{\bar kj}$ is a Hermitian metric on the holomorphic tangent bundle $T^{1,0}X$. Let $\Lambda$ denote the adjoint of wedging with $\o$. If $\eta$ is a $(1,1)$ form, then in coordinates $\Lambda\eta=-ig^{j\bar k}\eta_{\bar kj}$. The volume form on $X$ is given by $\dfrac{\o^n}{n!}$, and for simplicity we denote it by $\o^n$.

Given a metric $H$ on $E$, the curvature of this metric is the following endomorphism valued two-from:
\be
F:= \frac{i}{2}\,F_{\bar kj}{}^{\al}{}_{\gamma}\,dz^j\wedge d\bar z^k,\nonumber
\ee
where in a holomorphic frame $F_{\bar kj}{}^{\al}{}_{\gamma}=-\pl_{\bar k}(H^{\al\bar\b}\pl_j H_{\bar \b\gamma})$. We compute the degree of $E$ as follows:
\be
deg(E):=\int_X{\rm Tr}(\Lambda F)\,\o^n\nonumber.
\ee
Because $X$ is K\"ahler this definition is independent of the choice of metric. 
The slope of the vector bundle $E$ is defined to be the following quotient:
\be
\mu(E)=\frac{deg(E)}{rk(E)}.\nonumber
\ee
We define $E$ to be stable if for any proper torsion-free subsheaf ${\cal F}\subset E$, we have $
\mu({\cal F})<\mu (E).$
$E$ is semi-stable if $\mu({\cal F})\leq\mu (E)$ for all proper torsion-free subsheaves.

Since we are restricting ourselves to the case where $E$ is semi-stable but not stable, we can always assume there is at least one proper subsheaf $\cal F$ of $E$ such that $\mu({\cal F})=\mu(E)$. In general there may be many such subsheaves.

\begin{definition}

{\em Given a semi-stable bundle $E$, a }Seshadri filtration {\rm is a filtration of torsion free subsheaves }
\be
\label{SF1}
0\subset S^0\subset S^1\subset\cdots\subset S^q=E,
\ee
{\rm such that $\mu(S^i)=\mu(E)$ for all $i$, and each quotient $Q^i=S^i/S^{i-1}$ is torsion free and stable. }
\end{definition}
While such a filtration may not be unique, the direct sum of the stable quotients, denoted
$Gr^s(E):=\oplus_{i} Q^i$, is canonical, so given any two Seshadri filtrations the corresponding direct sums of quotients will be natually isomorphic $\cite{Kob}$.  We define the algebraic singular set of $E$ as
\be
Z_{alg}:=\{x\in X\,| Gr^s(E)_x {\rm \,\,is\,\, not\, \,free}\}.\nonumber
\ee
Since the sheaf $Gr^s(E)$ is torsion free, we know $Z_{alg}$ is of complex codimension two. 
\medskip

We now define the Donaldson heat flow. Not only is this flow the framework in which we realize the approximate Hermitian-Einstein structure on $E$, but this flow also provides our approach to the Yang-Mills flow. Fix an initial metric $H_0$ on $E$. Any other metric $H$ defines an endormphism $h\in Herm^+(E)$ by $h=H_0^{-1}H$. The Donaldson heat flow is a flow of endomorphisms $h=h(t)$ given by:
\be
h^{-1}\dot h=-(\Lambda F-\mu I),\nonumber
\ee
where $F$ is the curvature of the metric $H(t)=H_0h(t)$. We set the initial condition $h(0)=I$. A unique smooth solution of the flow exists for all $t\in[0,\infty)$, and on any stable bundle this solution will converge to a smooth Hermitian-Einstein metric $\cite{Don1},\cite{Don2},\cite{Simp},\cite{Siu}$. However, in our case $E$ is only semi-stable, so we do not expect the flow to converge. We do know from $\cite{J}$ that $E$ admits an approximate Hermitian-Einstein structure, so for all $\epsilon>0$ there exists a $t_\epsilon$ along the flow such that for $t\geq t_\epsilon$ we have:
\be
\sup_X|\Lambda F-\mu I|\leq \epsilon.\nonumber
\ee
We now describe our approach to the Yang-Mills flow, which is the viewpoint taken by Donaldson in $\cite{Don1}$. For details we refer the reader to $\cite{Don1}$, and just present the setup here.

Fix the metric $H_0$. Working in a unitary frame, the connection $A_0$ induced by this metric will split into a $(1,0)$ component $A_0'$ and a $(0,1)$ component $A_0''$, giving us the decomposition $A_0=A_0'+A_0''$. Now, starting with our initial holomorphic structure $\bar\pl_0=\bar\pl+A_0''$,
we consider the flow of holomorphic structures
$\bar\pl_t=\bar\pl+A'',\nonumber$
where $A''$ is defined by the action of $w=h^{1/2}$ on $A_0''$. Explicitly, this action is given by:
\be
A''=wA_0''w^{-1}-\bar\pl ww^{-1}.\nonumber
\ee
Now, given this flow of holomorphic structures, we can use the fixed metric $H_0$ to define a flow of unitary connections $A=A(t)$. The curvature of $A$ can be computed by the following formula:
\be
F_A=\bar \pl A'+\pl A''+A'\wedge A''+A''\wedge A'.\nonumber
\ee
With this setup one can check that $A$ evolves along the Yang-Mills flow:
\be
\dot A=- d^*_{A}F_{A}.\nonumber
\ee
In order to avoid notational confusion we always refer to the curvature evolving along the Donaldson Heat flow as $F$, and the curvature evolving along the Yang-Mills flow as $F_{A}$. These two curvatures are related by the action of $w$:
\be
F_A=w\,F\, w^{-1}.\nonumber
\ee

\medskip

We conclude this section by stating the convergence result of Hong and Tian from $\cite{HT}$. Consider a sequence of connections $A_j$ evolving along the Yang-Mills flow. Then, on $X\backslash Z_{an}$, along a subsequence the connections $A_j$ converge in $C^\infty$, modulo unitary gauge transformations, to a Yang-Mills connection $A_\infty$.
Thus, always working on $X\backslash Z_{an}$,  we have a sequence of holomorphic structures $(E,\bar\pl_j)$ which converge in $C^\infty$ to a holomorphic structure on a (possibly) different bundle $(E_\infty,\bar\pl_\infty)$. Since $E$ is semi-stable we know that the limiting Yang-Mills connection $A_\infty$ on $E_\infty$ will be Hermitian-Einstein, and thus will decompose $E_\infty$ into a direct sum of stable bundles:
\be
E_\infty=Q^1_\infty\oplus Q^2_\infty\oplus\cdots \oplus Q^p_\infty\nonumber,
\ee
each admitting an induced smooth Hermitian-Einstein connection. 
\section{Convergence Results}

In this section we prove convergence of the Seshadri filtration of $E$ along the flow, as well as the convergence of the holomorphic inclusion maps which define these sheaves. Given a torsion free subsheaf  ${\cal F}\subset E$, in many instances it will be simpler to consider the projection $\pi:E\longrightarrow E$ defining $\cal F$, so we go over this equivalence first.

Recall we have a fixed metric $H_0$ on $E$. Since we can view $\cal F$ as a holomorphic vector bundle off the singular set $\Sigma$ where $\cal F$ fails to be locally free, off $\Sigma$ we can define the orthogonal projection $\pi$. One can check that although $\pi$ develops a singularity on $\Sigma$, this projection is still in $L^2_1(X)$ $\cite{J}$. Conversely, in $\cite{UY}$ Uhlenbeck and Yau prove that given a $L^2_1$ projection $\pi$, it defines a coherent subsheaf of $(E,\bar\pl_0)$ provided it satisfies the following ``weakly holomorphic" condition:
\be
||(I-\pi)\,\bar\pl_0\,\pi||_{L^2}=0.\nonumber
\ee
Thus we can go back and forth between a cohearent subsheaf and its $L^2_1$ projection. 

 Let the following filtration of sheaves
\be
\label{SF}
0\subset \pi_0^0\subset\pi_0^1\subset\pi_0^2\cdots\subset\pi_0^p\subset E
\ee
be a Seshadri filtration of $E$.  The action of $w_j$ produces a sequence of filtrations $\{\pi^{(i)}_j\}$, where each $\pi^{(i)}_j$ is defined by orthogonal projection onto the subsheaf $w_j(\pi^{(i)}_0)$. Our first goal is to show that this sequence of filtrations converges along a subsequence. 

\begin{proposition}
\label{convergence of pi}
The sequence of Seshadri filtrations $\{\pi^{(i)}_j\}$ given by the action of $w_j$ converges in $L^2_1$ (along a subsequence), to a filtration $\{\ti\pi^{(i)}_\infty\}$ of $E_\infty$ off $Z_{an}$. 
\end{proposition}
\begin{proof}
For notational simplicity, we work with the subsheaf of lowest rank from $\eqref{SF}$ and denote it as $\pi_0$. We consider the sequence of projections $\pi_j$ defined by the action of $w_j$, and the result for the entire filtration will follow by induction on rank. First we show that the second fundamental form $\bar\pl_j\pi_j$ converges to zero in $L^2$. To see this, we note the following modification of the Chern-Weil formula:
\be
\label{CW}
\mu(\pi_j)=\mu(E)+\dfrac{1}{rk(\pi_j)}(\int_X{\rm Tr}((\Lambda F_{A_j}-\mu I)\circ \pi_j)\o^n-||\bar\pl_j\pi_j||_{L^2}^2).
\ee 
Now since $\mu(E)=\mu(\pi_j)$, we have that:
\be
||\bar\pl_j\pi_j||_{L^2}^2=\int_X{\rm Tr}((\Lambda F_{A_j}-\mu I)\circ \pi_j)\,\o^n.\nonumber
\ee
Recall that $F_{A_j}=w_j\circ F_j\circ w_j^{-1}$, where $F_j$ evolves along the Donaldson heat flow. This allows us to write:
\bea
\label {L1est}
||\bar\pl_j\pi_j||_{L^2}^2&=&\int_X{\rm Tr}((\Lambda F_j-\mu I)\circ w_j^{-1}\pi_jw_j)\,\o^n.\nonumber\\
&\leq&\int_X|\Lambda F_j-\mu I|\,\o^n,
\eea
where the inequality follows from the fact that $(w_j^{-1}\pi_jw_j)^2=w_j^{-1}\pi_jw_j$, thus the operator $w_j^{-1}\pi_jw_j$ has eigenvalues equal to only zero or one. Now since $E$ admits an apporximate Hermitian-Einstein structure, we see $||\bar\pl_j\pi_j||_{L^2}^2$ goes to zero as $j\rightarrow\infty$. Beacuse $\pi_j=\pi_j^*$ it follows that $|\bar\pl_j\pi_j|^2=|\pl_j\pi_j|^2$, thus we have $\pl_j\pi_j$ is uniformly bounded in $L^2$ and and $\pi_j$ converges along a subsequence to a weak limit $\ti\pi_\infty$ in $L^2_1$. We now work on a compact set $K$ away from $Z_{an}$. Since 
\be
\bar\pl_\infty\pi_j=\bar\pl_j\pi_j+(\bar\pl_\infty-\bar\pl_j)\pi_j\nonumber,
\ee
it follows that
\bea
||\bar\pl_\infty\pi_j||_{L^2(K)}&\leq&||\bar\pl_j\pi_j||_{L^2(K)}+||(\bar\pl_\infty-\bar\pl_j)\pi_j||_{L^2(K)}\nonumber\\
&\leq&||\bar\pl_j\pi_j||_{L^2(K)}+||A_j-A_\infty||_{L^q(K)}^\lambda||\pi_j||^{1-\lambda}_{L^r(K)},\nonumber
\eea
where $q$, $r$, and $\lambda$ are given by Holder's inequality. We have that $A_j\rightarrow A_\infty$ in $L^q$ on compact subsets $K$ away from $Z_{an}$ (by Theorem 3.6 from $\cite{U}$, and $\cite{HT}$). Because $||\bar\pl_j\pi_j||_{L^2}\rightarrow 0$ it follows that $||\bar\pl_\infty\pi_j||_{L^2(K)}\rightarrow 0$. Finally, from the simple formula
\be
\bar\pl_\infty\ti\pi_\infty=\bar\pl_\infty\pi_j+\bar\pl_\infty(\ti\pi_\infty-\pi_j),\nonumber
\ee
we see that
\bea
||\bar\pl_\infty\ti\pi_\infty||_{L^2(K)}&\leq&||\bar\pl_\infty\pi_j||_{L^2(K)}+||\bar\pl_\infty(\ti\pi_\infty-\pi_j)||_{L^2(K)}\nonumber\\
&=&||\bar\pl_\infty\pi_j||_{L^2(K)}+||\ti\pi_\infty-\pi_j||_{L^2_1(K)}.\nonumber
\eea
The left hand side is independent of $j$, so we would like to send $j$ to infinity proving $||\bar\pl_\infty\ti\pi_\infty||_{L^2(K)}=0$. We have to be careful about the second term on the right since $\pi_j$ only converges to $\ti\pi_\infty$ weakly in $L^2_1$. However, we can achieve strong $L^2_1$ convergence along a subsequence, since given how curvature decomposes onto subbundles $\pi_j$ with quotient $Q_j$ (for instance see $\cite{GH}$), we have:
\be
\int_K|\nabla_j (\bar\pl_j\pi_j)|^2\o^n\leq\int_K|F_{A_j}|^2\o^n\leq C,\nonumber
\ee
where $\nabla_j$ is the induced connection on $Hom(Q_j,\pi_j)$. Because the Yang-Mills energy is decreasing, we know $\pi_j$ is bounded in $L^2_2$, and thus along a subsequence we have strong convergence in $L^2_1$. It follows that $||\bar\pl_\infty\ti\pi_\infty||_{L^2(K)}=0$. This holds independent of which compact set $K$ we choose, so 
\be
||\bar\pl_\infty\ti\pi_\infty||_{L^2(X\backslash Z_{an})}=||\bar\pl_\infty\ti\pi_\infty||_{L^2(X)}=0,\nonumber
\ee
since $Z_{an}$ has complex codimension two. Thus $\ti\pi_\infty$ is a weakly holomorphic $L^2_1$ 
subbundle of $(E_\infty,\bar\pl_\infty)$. Furthermore, because the eigenvalues of the projections $\pi_j$ are either zero or one, we know that rk$(\ti\pi_\infty)$=rk$(\pi_j)$. It also follows that $\mu(\pi_0)=\mu(\ti\pi_\infty)$. To see this note that $\pi_0$ was chosen so that $\mu(\pi_0)=\mu(E)=\mu(E_\infty)$, and because $||\bar\pl_\infty\ti\pi_\infty||_{L^2}=0$ we know by formula $\eqref{CW}$ and the fact that $E_\infty$ admits an admissible Hermitian-Einstein metric that $\mu(\ti\pi_\infty)=\mu(E_\infty)$. 

\end{proof}

Now that we have established convergence of the projections $\eqref{SF}$ along the action of $w_j$, the next step is to show convergence of the holomorphic inclusion maps into $E$. For notational simplicity we let $S$ be the sheaf of lowest rank from $\eqref{SF1}$, and let $f_0$ be the holomorphic inclusion of $S$ into $E$:
\be
\label{holomap}
0\longrightarrow S\xrightarrow{{\phantom {X}}{f_0}{\phantom {X}}}(E,\bar\pl_0)\longrightarrow Q\longrightarrow 0.
\ee
By holomorphic inclusion we mean $f_0$ is a holomorphic section of the sheaf $Hom(S,(E,\bar\pl_0))$, where $(E,\bar\pl_0)$ denotes the bundle $E$ with the holomorphic structure $\bar\pl_0$. For the rest of this section we work away from $Z_{alg}$, so $S$ is a vector bundle. Note the composition $f_j:= w_j\circ f_0$ defines an inclusion of $S$ into $(E,\bar\pl_j)$. 
\begin{lemma}
For all $j$, the map $f_j$ is a holomorphic section of the sheaf $Hom(S,(E,\bar\pl_j))$.
\end{lemma}
\begin{proof}
Using the defining equation for the action of $ w_j$, we know that $
A''_j= w_jA_0''w_j^{-1}-\bar\pl w_j w_j^{-1}$. 
Multiplying on the right by $ w_j\circ f_0$ and rearranging terms we get the equation:
\be
\label{eqq}
(\bar\pl  w_j)f_0+A''_j w_jf_0- w_j A''_0f_0=0.
\ee
Now $f_0$ is holomorphic with respect to $\bar \pl_0$, thus it satisfies $\bar\pl f_0+A''_0f_0-f_0A_0''=0$, which we can write as
$A''_0f_0=f_0A''_0-\bar\pl f_0$.
Plugging this into equation $\eqref{eqq}$ proves that $f_j$ solves
\be
\label{eqqq}
\bar\pl_jf_j=\bar\pl f_j+A''_jf_j-f_jA''_0=0,
\ee
and is indeed holomorphic.
\end{proof}

We are now ready to show convergence of the maps $f_j$ along a subsequence. The main difficulty is that we do not have control of the connections $A_j$ as we approach the analytic singular set $Z_{an}$, forcing us to work on compact subsets off $Z_{an}$. Also, to get the desired convergence we need to choose a suitable normalization of $f_j$. This does not affect our result since in the end we just need to construct an isomorphism. Here we note that in the argument to follow the constant $C(r)$ may change from line to line, yet it is always universal in $j$. 

\begin{proposition}
Let $f_j=w_j\circ f_0$ be the sequence of holomorphic inclusions of $S$ into $(E,\bar\pl_j)$ constructed above. After choosing a suitable normalization and passing to a subsequence, this sequence converges on compact subsets of $X\backslash (Z_{an}\cup Z_{alg})$ in $L^p_{2}$ to a holomorphic map $f_\infty$, which is non-trivial. $f_\infty$ extends to a map on all of $X$. 
\end{proposition}
\begin{proof}

Let $Z=Z_{alg}\cup Z_{an}$ and consider the family of compact subsets $K(r)\subset X\backslash Z$, defined to be the complement of a tube around $Z$:
\be
K(r):=X\backslash \bigcup_{x\in Z}B_r(x).\nonumber
\ee
Fix $r_0>0$. We normalize $f_j$ by its supremum on the set $K(r_0)$, and call this normalization $\ti f_j$. Now, because equation $\eqref{eqqq}$ is independent of normalization, it follows that $\ti f_j$ solves:
\be
\bar\pl \ti f_j+A''_j\ti f_j-\ti f_jA_0''=0.\nonumber
\ee
By Hong-Tian \cite{HT} we know $A_j$ is bounded in $C^\infty$ on compact sets away from $Z_{an}$, in particular we apply their result to the set $K(r)$, for $r<r_0$. Thus on $K(r)$ we have:
\be
\label{bound}
|\bar\pl\ti f_j|\leq C(r) |\ti f_j|,
\ee
and by Kato's inequality:
\be
\nonumber
\bar\pl |\ti f_j|\leq C(r) |\ti f_j|.
\ee
Since $|\ti f_j|$ is a real function on $X$ we have $d|\ti f_j|\leq C(r)|\ti f_j|$. It follows that in any small open set $|\ti f_j|$ grows at most exponentially in distance as we travel away from $K(r_0)$. Because the boundary of $K(r)$ is a finite distance from $K(r_0)$, the fact that $|\ti f_j|\leq 1$ on $K(r_0)$ proves that $|\ti f_j|\leq C(r)$ on $K(r)$. This holds for any $r>0$, although it may degenerate as $r$ approaches $0$. 

We now have a $C^0$ bound for $|\ti f_j|$, which along with \eqref{bound} allows us to conclude $|\bar\pl\ti f_j|\leq C(r)$. In a similar argument to before, we use how the $(1,0)$ part of the connection $A_j$ evolves under the Yang-Mills flow: $\pl w_j=A_0' w_j-w_j A_j'$, along with the bounds from \cite{HT} and the fact that the coordinate derivatives of $f_0$ are fixed and bounded to obtain $|\pl \ti f_j|\leq C(r)$. Thus the $C^1$ norm of $\ti f_j$ is bounded uniformly by $C(r)$ on $K(r)$.

Now $\ti f_j$ solves the equation:
\be
\label{holj}
\bar\pl^\dagger(\bar\pl \ti f_j+A''_j\ti f_j-\ti f_jA''_0)=0,\nonumber
\ee
which we rearrange to get the following equality:
\be
\bar\pl^\dagger\bar\pl \ti f_j=\bar\pl^\dagger(A''_j\ti f_j-\ti f_jA''_0).\nonumber
\ee
For a fixed $r$ we have uniform $C^1$ bounds for $A_j$ and $\ti f_j$ on $K(r)$ and thus the right hand side is bounded in $C^0$ independent of $j$. Ellipticity of $\bar\pl^\dagger\bar\pl$ gives that $\ti f_j$ are uniformly bounded in $L^p_2$ for any $p$, and thus, still working on $K(r)$, the $\ti f_j$ converge weakly in $L^p_2$ (and strongly in $L^p_1$) to a limiting map $f_\infty$. 
We see that:
 \be
 \bar\pl_\infty f_\infty=\bar\pl_j(f_\infty-\ti f_j)-(\bar\pl_j-\bar\pl_\infty)f_\infty,\nonumber
 \ee
so it follows:
\be
|| \bar\pl_\infty f_\infty||_{L^p(K(r))}\leq||f_\infty-\ti f_j||_{L^p_1(K(r))}+||A_j-A_\infty||_{L^q(K(r))}^\lambda||f_\infty||^{1-\lambda}_{L^s(K(r))},\nonumber
\ee
where $q$, $s$, and $\lambda$ are given by Holder's inequality.  The left hand side is independent of $j$, so sending $j$ to infinity we see
\be
|| \bar\pl_\infty f_\infty||_{L^p(K(r))}=0\nonumber
\ee
for any $p$. By elliptic regularity we have that $f_\infty$ is smooth, and thus holomorphic.

Next we note $f_\infty$ is not identically zero on $K(r)$. This follows because $K(r_0)\subset K(r)$ and by our normalization there exists a sequence of points $x_j$ in $K(r_0)$ where $|\ti f_j|=1$. Combining this with the uniform $C^1$ bounds in a neighborhood of $x_j$, implies that the $L^2$ norm of $\ti f_j$ can not degenerate to zero.  

In fact we can show $f_\infty$ is holomorphic section of $Hom(S,E_\infty)$ on all of $X$. Pick any point $x_0\in X\backslash Z$. Then there exists a $r'<r$ such that $K(r')$ contains $x_0$. By choosing the sequence $\ti f_j$ from above, and repeating the convergence argument for the compact set $K(r')$ as opposed to $K(r)$, we get convergence along a subsequence to a new holomorphic map $f'_\infty$ defined on all of $K(r')$. Since we choose a subsequence of our original sequence we know $f'_\infty=f_\infty$ on $K(r)$, thus $f_\infty$ extends to a holomorphic map on all of $K(r')$. We can do this for each point in in $X\backslash Z_{an}$, thus $f_\infty$ is holomorphic everywhere in $X\backslash Z$. Since $H^{2n-2}(Z)=0$, by Lemma 3 in $\cite{Shiff}$ and Proposition 5.21 in \cite{Kob} we know $f_\infty$ extends to a holomorphic section on all of $X$.

\end{proof}

\section{Construction of an isomorphism}
As stated in Section $\ref{bkg}$, we know the limiting Yang-Mills connection $A_\infty$ on $E_\infty$ over $X\backslash Z_{an}$ is Hermitian-Einstein, and thus will decompose $E_\infty$ into a direct sum of stable bundles:
\be
\label{SFinf1}
E_\infty=Q^1_\infty\oplus Q^2_\infty\oplus\cdots \oplus Q^p_\infty,
\ee
each admitting an induced smooth Hermitian-Einstein connection. 

\begin{proposition}
\label{iso}
Off $Z_{an}$ there exists an isomorphism between $E_\infty$ and $Gr^s(E)$.
\end{proposition}
\begin{proof}
Recall that $Gr^s(E)$ is the direct sum of the stable quotients from the filtration $\eqref{SF}$. We will show each quotient $Q^{(i)}_0$ is isomorphic to its corresponding limit $\ti Q^{(i)}_\infty$ from Proposition $\ref {convergence of pi}$, working by induction on the rank of $\pi^{(i)}_0$. This will prove $\{\ti\pi^{(i)}_\infty\}$ forms a Seshadri filtration of $E_\infty$, thus $\oplus_i\ti Q^i_\infty$ must be isomorphic to $\oplus_iQ^i_\infty$.

Choose the destabilizing subsheaf of $(E,\pl_0)$ with the lowest rank, which we call $\pi_0$ for notational simplicity. Let $S$ be the holomorphic subbundle of $E$ defined by the projection $\pi_0$, and let $f_0$ denote the holomorphic inclusion of $S$ into $E$:
\be
0\longrightarrow S\xrightarrow{{\phantom {X}}{f_0}{\phantom {X}}}(E,\bar\pl_0)\longrightarrow Q\longrightarrow 0. \nonumber
\ee
Now, choose a sequence of holomorphic structures $\{\bar\pl_j\}$ evolving along by the Yang-Mills flow, and by the previous two propositions we can construct a limiting projection $\ti \pi_\infty$ and a non-zero limiting holomorphic map $f_\infty$. As before we know rk$(\pi_0)=$ rk$(\ti\pi_\infty)$, and $\ti\pi_\infty$ defines a coherent sub-sheaf $\ti S_\infty$ of $E_\infty$. Now, because along the flow $\pi_j\circ\ti f_j=\ti f_j$, we see that in the limit $\ti\pi_\infty\circ f_\infty=f_\infty.$ Thus we know $f_\infty$ is a holomorphic map not just into $E_\infty$, but one which includes into $\ti S_\infty$:
\be
f_\infty:S\longrightarrow \ti S_\infty.\nonumber
\ee
We also know that $\mu(S)=\mu(\ti S_\infty)$. Now, because we choose $S$ to have minimal rank among all subsheaves of $E$ with that property $\mu(S)=\mu(E)$, we know $S$ is stable. Thus by $\cite{Kob}$ Chapter V (7.11), since $S$ is stable and $\ti S_\infty$ semi-stable, the holomorphic map $f_\infty$ is injective and rk$(S)=$rk$(f_\infty(S))$. Thus by Corollary 7.12 from $\cite{Kob}$ the map $f_\infty$ is an isomorphism between $S$ and $\ti S_\infty$. We note that although the quoted results from $\cite{Kob}$ are only stated for compact K\"ahler manifolds, the arguments needed trivially carry over for the manifold $X\backslash Z_{an}$.  It now follows that $\ti S_\infty$ is stable and indecomposable. It is indecomposable because $S$ is indecomposable. To show $\ti S_\infty$ is stable assume there exists a subsheaf $\cal F$ such that $\mu({\cal F})=\mu(\ti S_\infty)$. Since $\ti S_\infty$ admits an admissible Hermitian-Einstein metric, by formula $\eqref {CW}$ the second fundamental form associated to ${\cal F}$ is zero, creating a holomorphic splitting of $\ti S_\infty$, which contradicts the fact that it is indecomposable.

We now continue this process. Since $S$ is stable we know its quotient $Q$ is semi-stable, and that along the Yang Mills flow the holomorphic structure $(Q,\bar\pl_j)$ converges to a limiting holomorphic structure on $(\ti Q_\infty,\bar\pl_\infty)$. If $p$ denotes the holomorphic projection from $E$ onto $Q$, and $p^\dagger$ is the adjoint of $p$ in the fixed metric $H_0$, then the sequence of induced connections on $Q$ is given by:
\be
p\circ(d+A_j)\circ p^\dagger\nonumber.
\ee
From this formula it is clear that these induced connections satisfy the same bounds as $A_j$ and converge on compact subsets away from $Z_{an}$ along a subsequence. The final thing we need to check in order to repeat the argument is that the second fundamental form associated to any destabilizing subsheaf of $Q$ goes to zero in $L^2$, and by estimate $\eqref{L1est}$ we see this follows if $||\Lambda F^Q_j-\mu I||_{L^1}$ goes to zero, where $F^Q_j$ is the curvature of the induced connection on $Q$. Using the well know decomposition formula of curvature onto quotient bundles (for instance see $\cite{GH}$) we have:
\be
F^Q_j=F|_Q+(\bar\pl_j\pi_j)^\dagger\wedge(\bar\pl_j\pi_j).\nonumber
\ee
Thus:
\be
||\Lambda F^Q_j-\mu I||_{L^1}\leq||\Lambda F_j-\mu I||_{L^1}+||\bar\pl_j\pi_j||^2_{L^2}.\nonumber
\ee
Now by estimate $\eqref{L1est}$ and the existence of an approximate Hermitian-Einstein structure it follows that $||\Lambda F^Q_j-\mu I||_{L^1}$ goes to zero. Thus we can pick a destabilizing subsheaf of $Q$ and the argument can be repeated inductively, allowing us to construct an isomorphism from each quotient $Q^i$ in $Gr^s(E)$ into one of the indecomposable stable bundles from
\be
\label{SFinf}
E_\infty=\ti Q^1_\infty\oplus\cdots\oplus\ti Q^p_\infty.
\ee
Since each $\ti Q^i_\infty$ is stable, this proves the $\ti S^i_\infty$ form a Seshadri filtraton of $E_\infty$, and thus $\eqref{SFinf}$ must be isomorphic to $\eqref{SFinf1}$.

\end{proof}
We have the following corollary, which we will need in the proof of Theorem $\ref{main theorem}$.
\begin{corollary}
\label{cor}
The algebraic singular set $Z_{alg}$ is contained in the analytic singular set $Z_{an}$.
\end{corollary}

\begin{proof}

We prove $Z_{alg}\subseteq Z_{an}$. Suppose there exists a point $x_0\in Z_{alg}$ which is not in $Z_{an}$. We know there exists a quotient $Q^i$ from $Gr^s(E)$ such that $Q^i$ is not locally free at $x_0$. Yet by Theorem $\ref{main theorem}$ we know $Q^i$ is isomorphic to some $Q^i_\infty$ from the direct sum $E_\infty=\oplus_i Q^i_\infty$, and since $E_\infty$ is a vector bundle off $Z_{an}$ we know $Q^i_\infty$ is locally free there. 

\end{proof}

It would be quite valuable to know the other set inclusion, proving that in fact $Z_{alg}= Z_{an}$. This would show that the bubbling set $Z_{an}$ is unique and canonical, and does not depend on the subsequence we choose along the Yang-Mills flow.  However, to do so a much more detailed analysis of the singular set is needed. 

We now extend the isomorphism over $Z_{an}$, finishing the proof of Theorem $\ref{main theorem}$:

\begin{proof}
First we note that $E_\infty$ can be extended over $Z_{an}$ as the reflexive sheaf $Gr^s(E)^{**}$. Colollary $\ref{cor}$ implies\be
\label{finalcong}
\Gamma(X\backslash Z_{an},Gr^s(E))\cong\Gamma(X\backslash Z_{an},Gr^s(E)^{**}),
\ee
since $Gr^s(E)$ is locally free on $X\backslash Z_{an}$. Also, because $Gr^s(E)^{**}$ is reflexive it is defined by $Hom(Gr^s(E)^{*},\cO)$. Since all holomorphic functions can be extended over $Z_{an}$ $\cite{Shiff}$, it follows that 
\be
\Gamma(X\backslash Z_{an},Gr^s(E)^{**})\cong\Gamma(X,Gr^s(E)^{**}).\nonumber
\ee
Combining this isomorphism with $\eqref{finalcong}$ and Proposition $\ref{iso}$ we see that $E_\infty$ extends over the singular set $Z_{an}$ as the reflexive sheaf $Gr^s(E)^{**}$. We now quickly go over the sheaf extension of Bando and Siu. 

As stated in $\cite{BaS}$, this extension is a consequence of Bando's removable singularity theorem $\cite{Ba}$ and Siu's slicing theorem $\cite{Siu1}$. We work on a coordinate chart $U\subset \C^n$. It follows from the final corollary in $\cite {Shiff}$ that given a point $x\in U$, since $H^{2n-4}(Z_{an})\leq C$ almost all complex two planes through $x$ intersect $Z_{an}$ at a finite number of points. Thus if we choose a countable dense set $A\subset U$, for each point in $A$ we can choose a complex plane with a fixed normal vector that intersects $Z_{an}$ at a finite number of points. Change coordinates so that this plane $P$ is given by $z_3=z_4=\cdots=z_n=c$. Let $\Delta^2$ be a polydisk in $P$. Given a domain $D\subset U$ which lies in $U\cap\C^{n-2}$, let $A'$ be the projection of $A$ onto $D$. We are now ready to apply the slicing theorem. $E_\infty$ is a locally free sheaf defined on $(D\times \Delta^2)\cap Z_{an}$. For each $t\in A'$, we have that $(\{t\}\times \Delta^2)\cap Z_{an}$ is a finite number of points, thus by Theorem 10 from $\cite{Ba}$ $E_\infty$ can be extended to a locally free sheaf $E_\infty(t)$ on $\{t\}\times\Delta^2$. Thus $E_\infty$ can be uniquely extended to a reflexive sheaf $\hat E_\infty$ on $D\times \Delta^2$. 

The uniqueness condition stated in $\cite{Siu1}$ is characterized by the fact that given any other reflexive extension (in our case $Gr^s(E)^{**}$), there exists a sheaf isomorphsim $\phi:\hat E_\infty\longrightarrow Gr^s(E)^{**}$ on $X$, which restricts to the isomorphism constructed in Proposition $\ref{iso}$ on $X\backslash Z_{an}$. This completes the proof of Theorem $\ref{main theorem}$.

\end{proof}

Thus even through we do not know whether $Z_{an}$ depends on the subsequence $A_j$, the limiting reflexive sheaf $\hat E_\infty$ (defined on all of $X$) is canonical and does not depend on the choice of subsequence. 

\end{normalsize}

\vspace{10mm}

\begin{centering}

\textnormal{ Department of Mathematics,
Harvard University,
Cambridge, MA 02138\\
e-mail: ajacob@math.harvard.edu}

\end{centering}

\end{document}